\documentclass[11pt,a4paper]{amsart}
\usepackage{amsfonts,amssymb,amscd,amsmath,enumerate,verbatim,calc}
\usepackage{mathrsfs}
\numberwithin{equation}{section}
\newtheorem{thm}{Theorem}[section]
\newtheorem{cor}[thm]{Corollary}
\newtheorem{lem}[thm]{Lemma}
\newtheorem{prop}[thm]{Proposition}
\newtheorem{defn}[thm]{Definition}

\newtheorem{rem}[thm]{Remark}

\DeclareMathOperator{\Ext}{Ext} \DeclareMathOperator{\Supp}{Supp}
\DeclareMathOperator{\V}{V} \DeclareMathOperator{\Hom}{Hom}
\DeclareMathOperator{\Ker}{Ker} \DeclareMathOperator{\Coker}{Coker}
\DeclareMathOperator{\Image}{Im} 
\DeclareMathOperator{\cd}{cd} 

 \DeclareMathOperator{\Max}{Max}
\DeclareMathOperator{\lc}{H} 
 \DeclareMathOperator{\Spec}{Spec}
\DeclareMathOperator{\G}{\Gamma}

\DeclareMathOperator{\ara}{ara}

\newcommand{\fa}{\mathfrak{a}}
\newcommand{\fb}{\mathfrak{b}}

\newcommand{\fm}{\mathfrak{m}}
\newcommand{\fp}{\mathfrak{p}}

\newcommand{\lo}{\longrightarrow}

\begin{document}

\title[Abelian category of cominimax modules and local cohomology]
{Abelian category of cominimax modules and local cohomology}

\bibliographystyle{amsplain}


   \author[M. Aghapournahr]{Moharram Aghapournahr}
\address{Department of Mathematics, Faculty of Science, Arak University,
Arak, 38156-8-8349, Iran.}
\email{m.aghapour@gmail.com}






\begin{abstract}
Let $R$ be a commutative Noetherian ring, $\fa$ an ideal of $R$,
 $M$ an arbitrary $R$-module and $X$ a finite $R$-module. We prove that the category of $\fa$-cominimax modules is a Melkersson subcategory of $R$-modules whenever $\dim R\leq 1$ and is an Abelian subcategory whenever $\dim R\leq 2$. We prove a characterization theorem for $\lc_{\fa}^{i}(M)$ and $\lc_{\fa}^{i}(X,M)$ to be $\fa$-cominimax for all $i$, whenever one of the following cases holds:
(a) $\ara (\fa)\leq 1$, (b) $\dim R/\fa \leq 1$ or (c) $\dim R\leq 2$.
\end{abstract}



\maketitle


\section{Introduction}

Throughout  this paper $R$ is a commutative Noetherian ring with non-zero
identity and $\fa$ an ideal of $R$. For an $R$-module $M$,  the $i^{th}$
local cohomology module $M$ with respect to ideal $\fa$ is defined as
\begin{center}
$\lc^{i}_{\fa}(M) \cong \underset{n}\varinjlim \Ext^{i}_{R}(R/{\fa}^{n},M).$
\end{center}

Also the generalized local
	cohomology module
	\begin{center}
		$\lc^{i}_{\fa}(X,M) \cong \underset{n}\varinjlim
		\Ext^{i}_{R}(X/{\fa}^{n}X,M).$
	\end{center}
	for all $R$--modules $M$ and $X$ was introduced by Herzog in \cite{He}.
	Clearly it is a generalization of ordinary local cohomology module.

We refer the reader to \cite{BSh} for more details about the local cohomology.

In \cite{Gro}, Grothendieck conjectured that for any ideal $\fa$ of $R$ and any finitely generated $R$-module $M$,
$\Hom_R(R/\fa , \lc^{i}_{\fa}(M))$ is a finitely generated $R$-module for all $i$.  Hartshorne \cite{Har} provided a
counterexample to Grothendieck’s conjecture. He defined an $R$-module $M$ to be $\fa$--{\it cofinite} if $\Supp_R(M)\subseteq \V(\fa)$ and $\Ext^{j}_{R}(R/{\fa},M)$ are finitely generated for all $j$ and he asked:\\

(i) \emph{ For which rings $R$ and ideals $\fa$ are the modules $\lc^{i}_{\fa}(M)$
	$\fa$-cofinite for all $i$ and all finitely generated
	modules $M$?}\\

(ii) \emph{ Whether the category $\mathscr{C}(R,\fa)_{cof}$
	of $\fa$-cofinite modules forms an Abelian subcategory of the category of all $R$-modules?
	That is, if $f: M\longrightarrow N$ is an $R$-homomorphism of
	$\fa$-cofinite modules, are $\Ker f$ and ${\Coker} f$ $\fa$-cofinite?}\\

With respect to the question (i), there are several papers devoted to this question; for example see \cite{HK, DM, MV, Mel, Mel1, BN2, BN, AB}.

With respect to the question (ii), Hartshorne with an example showed that this is
not true in general. However, it is proved in \cite[Theorem 2.2 (ii)]{DFT}, \cite[Theorem 2.6]{Mel1} and \cite[Theorem 7.4]{Mel} that the category $\mathscr{C}(R,\fa)_{cof}$
of $\fa$-cofinite modules forms an Abelian subcategory of the category of all $R$-modules respectively in cases $\cd({\fa},R)\leq 1$, $\dim R/{\fa}\leq 1$ and $\dim R\leq 2$.

Recall that a module $M$ is a \emph{minimax} module if there is a
finitely generated submodule $N$ of $M$ such that the quotient module $M/N$
is Artinian.  Minimax modules have been studied by
Z\"{o}schinger in
\cite{Zrmm}. Note that for a complete Noetherian local ring, the class of minimax modules is the same as the class of Matlis reflexive modules (see \cite{E} and \cite{Zi}).
Since the class of minimax modules is a generalization of Matlis reflexive modules, thus the study of minimax modules is as important as the study of Matlis reflexive modules. As a generalization of $\fa$-cofinite modules
in \cite{ANV}, the authors, introduced the concept of {\it $\fa$-cominimax} modules or cominimax modules with respect to $\fa$. An $R$-module $M$ is $\fa$-cominimax module if $\Supp_R(M)\subseteq \V(\fa)$ and $\Ext^{i}_{R}(R/\fa,M)$ is a minimax module for all $i$.

 Since the concept of minimax modules is a natural generalization of the finitely generated modules,
 many authors studied the minimaxness and cominimaxness of local cohomology modules and answered the Hartshorne's
 question in the class of minimax modules (see for example \cite{A3, A1, A2, AM, BN2, HR, BNS1, I}).

Recall also that a class of $R$-modules is said to be a {\it Serre subcategory of the category of $R$-modules}, when it is closed under taking submodules, quotients and extensions and a full subcategory $\mathcal{S}$ of the category of  $R$-modules is said to be { \it Melkersson  subcategory} with respect to the ideal $\fa$   if for any $\fa$-torsion  $R$-module $M$, $(0 :_{M} \fa)\in \mathcal{S}$ implies $M \in \mathcal{S}$ (see \cite{AM} and \cite{ATV}).\\ 
 
In this paper, in Section 2, we bring some preliminary results that we need to prove our main theorems. Now, it is natural to ask  whether the category $\mathscr{C}(R,\fa)_{comin}$
of $\fa$-cominimax modules forms an Abelian subcategory of the category of all $R$-modules? Among other things, in Section 3, We prepare an   affirmative answer to this question in the case $\dim R\leq 2$. More precisely we prove that:

\begin{prop}\textsl{{\rm(}See Proposition \ref{melc}{\rm)}}
	Let $R$ be a Noetherian ring with $\dim R\leq 1$. Let  $M$ be an  $R$-module such $\Supp_{R}(M) \subseteq \V(\fa)$ and $(0 :_{M} \fa)$  is a minimax $R$-modules. Then $M$
	is minimax.  In this case the class of $\fa$-cominimax $R$-modules is a 
	Serre subcategory of $R$-modules.
\end{prop}

\begin{thm}\textsl{{\rm(}See Theorem \ref{dim2}{\rm)}}
	Let $R$ be a Noetherin  ring  with $\dim R\leq 2$. Then for any ideal $\fa$ of $R$, the full subcategory of $\fa$-cominimax $R$-modules of the category of $R$-modules {\rm (}$\mathscr{C}(R,
	\fa)_{comin}${\rm )} is Abelian.
\end{thm}

In Section 4, we study relationship between cominimaxness of local cohomology and generalized local cohomology modules and we prove a characterization result (see Corollary \ref{min}), in the cases  (a) $\ara (\fa)\leq 1$, (b) $\dim R/\fa \leq 1$ or (c) $\dim R\leq 2$.\\

	Throughout this paper, $R$ will always be a commutative Noetherian ring with non-zero identity and $\fa$ and $\fb$ will be  ideals of $R$. We shall use $\Max R$ to denote the set of all maximal ideals of $R$. Also, for an ideal $\fa$ of $R$, we denote $\{\mathfrak p \in \Spec R:\, \mathfrak p\supseteq \fa \}$ by $V(\fa)$. For any unexplained notation and terminology we refer the reader to \cite{BSh}, \cite{BH} and \cite{Mat}.




\section{Prelaminaries}

We begin this section with some preliminaries which are needed in the proof of main results of the paper. The following remark is some elementary properties of the class of minimax $R$-modules which we shall use.

\begin{rem}\label{rem 2.2}
	The following statement holds:
	\begin{itemize}
		\item[(i)] The class of minimax modules contains all finitely generated and all Artinian modules;
		
		\item[(ii)] Let $0 \longrightarrow L \longrightarrow M \longrightarrow N \longrightarrow 0 $ be an exact sequence of 
		$R$-modules. Then $M$ is minimax if and  only if $L$ and $N$ are both minimax {\rm(}see \cite[Lemma 2.1]{BN2}{\rm)}. Thus any 
		submodule and quotient of a minimax module is minimax;
		
		\item[(iii)] The set of associated primes of any minimax $R$-module is finite;
		
		\item[(iv)] Every zero-dimensional minimax $R$-module is Artinian;
		
		\item[(v)] If $M$ is a minimax $R$-module and $\fp$ is a non-maximal prime ideal of $R$, then $M_{\fp}$ is a finitely
		generated $R_{\fp}$-module.
	\end{itemize}
\end{rem}

The following well-known lemma which is true even for an arbitrary Serre subcategory is needed frequently in this paper.

\begin{lem}\label{ser}
	If $M$ is an   $R$-module such that $(0 :_{M} \fa)$   is minimax, then so is $(0 :_{M} \fa^{n})$  for each $n$.
\end{lem}
\begin{proof}
	For each $n \geq 2$, there exists an exact sequence of modules
	\begin{center}
		$0 \rightarrow (0 :_{M} \fa) \rightarrow  (0 :_{M} \fa^{n}) \overset{f}\rightarrow  a_{1}(0 :_{M} \fa^{n}) \oplus ...\oplus a_{s}(0 :_{M} \fa^{n})$
	\end{center}
	where $\fa = (a_{1}, a_{2}, ... , a_{s})$ and $f(x) = (a_{1}x, ... , a_{s}x)$. Since $a_{i}(0 :_{M} \fa^{n})$ is a submodule of the
	minimax module $(0 :_{M} \fa^{n-1})$ for each $i$, the result is obtained by an easy induction on $n$.
\end{proof}

\begin{lem}\label{c}
Suppose $x\in \fa$ and $\Supp_R(M)\subseteq \V(\fa)$. If $0:_{M}x$ and $M/xM$ are both $\fa$-cominimax, then $M$ must also be $\fa$-cominimax.
\end{lem}

\begin{proof}
The proof is similar to the proof of \cite[Corollary 3.4]{Mel}.	
\end{proof}



\begin{lem} \label{1}
	Let $X$ be a finite $R$--module, $M$ be an arbitrary $R$--module. Then the following statements hold true.
	\begin{itemize}
		\item[(a)] $\Gamma_{\fa}(X, M)\cong \Hom_{R}(X, \Gamma_\fa (M)).$
		\item[(b)] If $\Supp_R(X)\cap \Supp_R(M)\subseteq \V(\fa)$, then $\lc^{i}_{\fa}(X, M)\cong\Ext^i_{R}(X, M)$ for all $i$.
	\end{itemize}
\end{lem}

\begin{proof}
See \cite[Lemma 2.5]{VA}.
\end{proof}

\begin{lem}\label{anv}
Let $R$ be a Noetherian  ring  with $\dim R\leq 2$. Let
$\fa$ be an ideal of $R$ and  $M$ a non-zero minimax
$R$-module. Then  $\lc^{i}_\fa(M)$ is $\fa$-cominimax for
all $ i \geq 0$.	
\end{lem}

\begin{proof}
It follows from  Grothendieck’s Vanishing Theorem \cite[Theorem 6.1.2]{BSh} that $\lc^{i}_\fa(M)=0$ for all $i\geq 3$. Also $\lc^{2}_\fa(M)$ is Artinian and $\lc^{0}_\fa(M)$ is minimax. So both of them are $\fa$-cominimax. Now the assertion follows form \cite[Corollary 3.9]{ANV}.  
\end{proof}


\section{An Abelian category  of cominimax modules}

\begin{defn}
	Let $\fa$ be an ideal of $R$, $n$ be a non-negative integer and $M$ an $R$-module. We say that $R$ admits $P^{\prime}_{n}(\fa)$
	if for any  $R$-module $M$, the following implication holds:
	
	If $Ext_{R}^{i}(R/\fa , M)$ is minimax for all $i \leq n$ and $\Supp (M) \subseteq \V(\fa)$,\\
	then $M$ is $\fa$-cominimax.
	\end{defn}

We now present the first main theorem of this section which is generalization of \cite[Theorem 2.3]{NS}.

\begin{thm}\label{2.2}
	Let $R$ be a Noetherian ring of dimension $d \geq 1$ admitting the condition $P^{\prime}_{d-1}(\fa)$ for all ideals $\fa$ of dimension $\leq d-1$  (i.e. $\dim R/ \fa \leq d-1$), then $R$ admits the condition $P^{\prime}_{d-1}(\fa)$ for all ideals $\fa$ of $R$.
\end{thm}
\begin{proof}
Assume that $M$ is an   $R$-module and $\fa$ is an arbitrary ideal such that $\Supp (M) \subseteq \V(\fa)$ and $Ext_{R}^{i}(R/\fa , M)$
is minimax for all $i \leq d-1$. We show that $M$ is $\fa$-cominimax. In the case where $\fa$ is nilpotent, say $\fa^{n}=0$
for some integer $n$, we have $M= (0 :_{M} \fa^{n})$. Now, since $ (0 :_{M} \fa)$ is minimax, Lemma \ref{ser} implies that
$M$ is  minimax and so $\V(\fa) = \Spec R$ forces that  $M $ is $\fa$-cominimax. Now suppose that $\fa$ is not nilpotent. In 
this case, we can choose a positive integer $n$ such that $(0 :_{R} \fa^{n}) = \Gamma_{\fa}(R)$. Put $\overline{R} = R/\Gamma_{\fa}(R)$   and $\overline{M} = M/(0 :_{M} \fa^{n})$ which is an $\overline{R}$-module. Taking $\overline{\fa}$ as an image of $\fa$ in 
$\overline{R}$, we have $\Gamma_{\overline{\fa}}(\overline{R}) = 0$. Thus  $\overline{\fa}$ contains an $\overline{R}$-regular element
so that $\dim R/\fa + \Gamma_{\fa}(R) = \dim \overline{R}/ \overline{\fa} \leq d-1$. The assumption on $M$ together with the fact that
$\Supp_{R}(R/ \fa + \Gamma_{\fa}(R)) \subset \Supp_{R}(R/ \fa )$ and \cite[Lemma 2.1]{A2} imply that
$\Ext_{R}^{i}(R/ \fa + \Gamma_{\fa}(R) , M)$ is minimax for all $i \leq d-1$. In view of Lemma \ref{ser}, the module 
$(0 :_{M} \fa^{n})$ is minimax and thus $\Ext_{R}^{i}(R/ \fa + \Gamma_{\fa}(R) , \overline{M})$ is minimax for all
$i \leq d-1$. On the other hand, it is clear that $\Supp(\overline{M}) \subseteq \V(\fa + \Gamma_{\fa}(R))$ and since by the assumption 
$R$ admits $P^{\prime}_{d-1}(\fa + \Gamma_{\fa}(R))$, the module $\overline{M}$ is $\fa + \Gamma_{\fa}(R)$-cominimax. Now using 
the change of ring principle \cite[Theorem 4.8]{ATV}, the module $M$ is an $\fa$-cominimax $R$-module; and finally the minimaxness of $(0 :_{M} \fa^{n})$ forces that $M$ is an $\fa$-cominimax $R$-module.
\end{proof}

\begin{cor}\label{cor 2.4}
	Let $R$ be a Noetherian ring of dimension $2$. Then $R$ admits the condition $P^{\prime}_{1}(\fa)$ for all ideals $\fa$ of $R$.
\end{cor}

\begin{proof}
	It follows from Lemma \cite[Lemma 2.4]{KA} that $R$ admits the condition $P^{\prime}_{1}(\fa)$ for all ideals $\fa$ with $\dim R/\fa =1$. Thus the result follows by Theorem \ref{2.2}.
\end{proof}
\begin{cor}\label{cor 4.4}
	Let $(R,\fm)$ be a Noetherian local ring of dimension $3$. Then $R$ admits the condition $P^{\prime}_{2}(\fa)$ for all ideals $\fa$ of $R$.
\end{cor}
\begin{proof}
	It follows from Theorem \cite[Theorem 2.6]{KA} that $R$ admits the condition $P^{\prime}_{2}(\fa)$ for all ideals $\fa$ with $\dim R/\fa =2$. Thus the result follows by Theorem \ref{2.2}.
\end{proof}




The following proposition shows that the class of minimax $R$-modules is a Melkersson subcategory of $R$-modules whenever $\dim R\leq 1$. Therefore, in this case, submodules and quotients modules of an $\fa$-cominimax modules are always $\fa$-cominimax.

\begin{prop}\label{melc}
	Let $R$ be a Noetherian ring with $\dim R\leq 1$. Let  $M$ be an  $R$-module such that $\Supp_{R}(M) \subseteq \V(\fa)$ and $(0 :_{M} \fa)$  is a minimax $R$-modules. Then $M$
	is minimax.  In this case the class of $\fa$-cominimax $R$-modules is a 
	Serre subcategory of $R$-modules.
\end{prop}

\begin{proof}
	If $\fa$ is nilpotent, then there is $n \in \mathbb{N}$ such that $\fa^{n}=0$, so $M = (0 :_{M} \fa^{n})$. By Lemma \ref{ser}, it follows
	that $M$ is minimax. Therefore suppose that $\fa$ is not nilpotent. Since $R$ is a Noetherian ring, thus there exists $n \in \mathbb{N}$ such
	that $ (0 :_{R}\fa^{n})=\Gamma_{\fa}(R)$. Then it is easy to see that $\overline{M}= \dfrac{M}{(0 :_{M}\fa^{n})}$ is a module over the
	ring $\overline{R}= \dfrac{R}{\Gamma_{\fa}(R)}$. Let $\overline{\fa}$ be the image of $\fa$ in $\overline{R}$. Then $\overline{\fa}$ contains
	an $\overline{R}$-regular element and therefore $\dim \dfrac{\overline{R}}{\overline{\fa}}=0$. Suppose $(0 :_{M} \fa)$ is minimax.
	Then, in view of Lemma \ref{ser}, it follows that $(0 :_{\overline{M}} \overline{\fa}) = (\dfrac{0 :_{{M}} {\fa}^{n+1}}{0 :_{{M}} {\fa}^{n}})$ is a minimax  $R$-module. Since $(0 :_{\overline{M}} \overline{\fa}) $ is an $\overline{R}$-module and
	$\overline{\fa}(0 :_{\overline{M}} \overline{\fa})=0$, hence $(0 :_{\overline{M}} \overline{\fa})$ has $\dfrac{\overline{R}}{\overline{\fa}}$-module structure. Since $\dfrac{\overline{R}}{\overline{\fa}}$ is an Artinian ring,  so 
	\begin{center}
		$\Supp_{{\overline{R}}/{\overline{\fa}}} (0 :_{\overline{M}} \overline{\fa}) \subseteq \Spec(\dfrac{\overline{R}}{\overline{\fa}})
		= \Max (\dfrac{\overline{R}}{\overline{\fa}})$.
	\end{center} 
	Thus by Remark \ref{rem 2.2} (iv), it follows that $(0 :_{\overline{M}} \overline{\fa})$  is a  Artinian $\dfrac{\overline{R}}{\overline{\fa}}$-module and therefore  Artinian $\overline{R}$-module. Since $\Supp_{R}(\overline{M}) \subseteq \Supp_{R}(M) \subseteq
	\V(\fa)$, it is easy to see that $\Supp_{\overline{R}}(\overline{M}) \subseteq  \V(\overline{\fa})$, then by Melkersson's theorem \cite[Theorem 7.1.2]{BSh}, $\overline{M}$
	is a  Artinian $\overline{R}$-module. So $\overline{M}$ is an  Artinian $R$-module. Now Lemma \ref{ser}, Remark  \ref{rem 2.2} (ii)  and the exact sequence 
	\begin{center}
		$0 \rightarrow (0 :_{M} \fa^{n}) \rightarrow M \rightarrow \overline{M} \rightarrow 0$
	\end{center}
	implies that $M$ is minimax as required.
\end{proof}

In the following corollary, we look at the case $\dim R = 2$. In this case we give a criterion for a quotient of an $\fa$-cominimax module to be $\fa$-cominimax.

\begin{cor}\label{d2}
	Let $R$ be a ring of dimension $2$. Let $M$ be an $\fa$-cominimax module and $N$ be a homomorphic image of $M$. Then $N$ is
	$\fa$-cominimax if and only if $(0 :_{N}\fa)$ is minimax.
\end{cor}

\begin{proof}
	Let $f: M \rightarrow N$ be a surjective $R$-linear map with $K = \Ker f$ and let $0 :_{N}\fa$ be minimax. By the assumption 
	$\Hom_{R}(A/\fa , K)$ and $\Ext_{R}^{1}(A/\fa , K)$ are minimax; and hence Corollary \ref{cor 2.4} implies that $K$ is 
	$\fa$-cominimax. Therefore $N$ is $\fa$-cominimax.  
\end{proof}

\begin{cor}\label{dim2}
	Let $R$ be a Noetherian  ring  with $\dim R\leq 2$. Then for any ideal $\fa$ of $R$, the full subcategory of $\fa$-cominimax $R$-modules of the category of $R$-modules {\rm (}$\mathscr{C}(R,
	\fa)_{comin}${\rm )} is Abelian.
\end{cor}

\begin{proof}
	In view of Proposition \ref{melc}, we may assume that $\dim R=2$. Now, let $f: M \rightarrow N$ be any 
	homomorphism between  $\fa$-cominimax modules such that $K = \Ker f , I = \Image f$ and $C = \Coker f$.  According to the Corollary \ref{d2}, since $0 :_{I}\fa$ is minimax, $I$ is $\fa$-cominimax; and therefore so is $K$ by the exact sequence 
	$0 \rightarrow K \rightarrow M \rightarrow I \rightarrow 0$. on the other hand, the exact sequence $0 \rightarrow I \rightarrow N \rightarrow C \rightarrow 0$ force that $C$ is $\fa$-cominimax.	
\end{proof}

\begin{cor}\label{d3}
	Let $R$ be a Noetherian  ring  with $\dim R\leq 2$. Let
	$\fa$ be an ideal of $R$ and	
	\begin{center} 	
		$ X^\bullet:\dots \lo X^{i} \lo X^{i+1} \lo X^{i+2}\lo \dots$
	\end{center} 	
	
	be a complex such that $X^{i}\in \mathscr{C}(R,
	\fa)_{comin}$ for all $i\in \Bbb Z$. Then the $i$th homology
	module $\lc^i(X^\bullet)$ is in $X^{i}\in \mathscr{C}(R,
	\fa)_{comin}$.
	
\end{cor}

\begin{proof}
	The assertion follows from Corollary \ref{dim2}.
\end{proof}


\begin{cor}
	\label{torext}
	Let $R$ be a Noetherian  ring  with $\dim R\leq 2$. Let
	$\fa$ be an ideal of $R$ and  $M$ is a non-zero $\fa$-cominimax $R$-module. Then, the R-modules ${\rm Ext}^{i}_R(N,M)$ and ${\rm Tor}^R_{i}(N,M)$ are
	$\fa$-cominimax $R$-modules, for all finitely generated $R$-modules $N$
	and all integers $ i \geq 0$.
\end{cor}
\proof Since $N$ is finitely generated it follows that $N$ has a free resolution of
finitely generated free modules. Now the assertion follows using Corollaries
\ref{dim2}, \ref{d3}  and computing the modules ${\rm
	Ext}^i_R(N,M)$ and ${\rm Tor}_i^R(N,M)$, by this
free resolution. \qed\\

\begin{cor}
	\label{torextlc} Let $R$ be a Noetherian  ring  with $\dim R\leq 2$. Let
	$\fa$ be an ideal of $R$ and  $M$ a non-zero minimax
	$R$-module. Then for
	each finite $R$-module $N$, the $R$-modules ${\rm Ext}^{j}_R(N,\lc^{i}_\fa(M))$ and
	${\rm Tor}^R_{j}(N,\lc^{i}_\fa(M))$ are $\fa$-cominimax for
	all $ i \geq 0$ and $ j \geq 0$.
\end{cor}
\proof Note that  by Lemma \ref{anv}, it follows that $\lc^{i}_\fa(M)$ is $\fa$-cominimax for
all $ i \geq 0$. Now the assertion follows from Corollary \ref{d3}.
\qed\\

It has been proved in  Corollary \ref{dim2}, \cite[Corollary 2.10]{A2} and \cite[Theorem 2.6]{I} that the full subcategory of $\fa$-cominimax $R$-modules of the category of $R$-modules is Abelian in the cases $\dim R\leq 2$, $\ara \fa=1$ and $\dim R/\fa\leq 1$. 
If $\ara \fa=1$ then $\cd(\fa,R)\leq 1$ but the converse is not true in general (see \cite[Example 2.3]{DFT}). We close this section by offering a question and problem for further research. The following question is at present far from being solved.\\

\noindent {\bf Question:} Let $R$ be a commutative Noetherian ring with non-zero identity and
$\fa$ an ideal of $R$ with $\cd(\fa,R)\leq 1$. Is $\mathscr {C}(R,\fa)_{comin}$ an Abelian full subcategory of $R$-modules?

\section{Cominimaxness of local cohomology and generalized local cohomology}

The following theorem is a generalization of \cite[Theorem 7.10]{Mel} to the class of minimax and cominimax modules.
\begin{thm}\label{T:HE2} 
	Let $R$ be a Noetherian ring with $\dim R\le 2$ and let 
	$\fa$ be a proper ideal and $M$ an $R$--module.
	The following conditions are equivalent:
	\begin{itemize}
		\item[(i)] ${\lc}^i_{\fa}(M)$ is $\fa$-cominimax, for all $i$.
		\item[(ii)] $\Ext^i_R(R/\fa ,M)$ is minimax for all $i$.
		\item[(iii)] $\Ext^i_R(R/\fa ,M)$ is minimax for $i\le 2$.
	\end{itemize} 
\end{thm}
\begin{proof}
	We need by \cite[Proposition 3.9]{Mel} just to show that (i) follows from (iii).
	Suppose that $M$ satisfies  (iii).  
	If $\fa$ is nilpotent, then it is easy to see that a module is $\fa$-cominimax, if and only if it is 
	minimax. If $\fa$ is non-nilpotent, take  
	$n$ such that $0\underset{R}{:}{\fa}^n = \G_{\fa}{(R)}$. 
	There is $x\in\fa$ which is regular  on $\overline R=R/{\G_{\fa}{(R)}}$, and therefore 
	$\dim{\overline R/{x\overline R}}\leq 1$.  
	The module $\overline M=M/{0\underset{M}{:}{\fa}^n}$ has a natural structure as a 
	module over $\overline R$. 
	Since $0\underset{M}{:}{\fa}^n$ is  minimax, 
	$\overline M$ must also satisfy (iii).
	The exact sequence 
	$0\lo{0\underset{M}{:}{\fa}^n}\lo M \lo\overline M\lo 0$ yields 
	the exact sequence 
	$0\lo{0\underset{M}{:}{\fa}^n}\lo\Gamma_{\fa}(M)\lo\Gamma_{\fa}(\overline M)\lo0$
	and  isomorphisms 
	$\lc^i_{\fa}(M)\cong{\lc^i_{\fa}(\overline M)}$ for $i \ge 1$.
	Thus replacing  $M$ by $\overline M$, 
	we may assume that $M$ is a module over $\overline R$. 
	Let  $L=\G_{\fa}(N)$, where $N=0\underset{M}{:}x\subset M$.
	Since $0\underset{L}{:}\fa=0\underset{M}{:}\fa$, which is minimax,
	Proposition \ref{melc} implies that   
	$L$ is minimax and so $\fa$-cominimax  and therefore satisfies  (ii). 
	From the exact sequence $0\lo N\lo M\lo xM\lo 0$,  we get that 
	$\Ext_R^1(R/\fa, N)$ is minimax. 
	Hence $\Ext_R^1(R/\fa,N/L)$ is minimax. 
	By \cite[Lemma 7.9]{Mel}, 
	$\Ext_R^1(R/\fa,N/L)\cong\Hom_R(R/\fa,\lc_\fa^1(N/L))$.
	Also $\lc_\fa^1(N)\cong\lc_\fa^1(N/L)$, so $\Hom_R(R/\fa,\lc_\fa^1(N))$ is minimax. 
	Hence by Proposition \ref{melc} the module $\lc_\fa^1(N)$ is minimax and so $\fa$-cominimax.
	Since $\lc_\fa^i(N)=0$ for $i>1$, \cite[Proposition 3.9]{Mel} implies that 
	$N=0\underset{M}{:}x$ satisfies (ii). 
	From the exactness of  $0\to N\to M\to xM\to 0$, we therefore get that 
	$\Ext_R^1(R/\fa, xM)$ and
	$\Ext_R^2(R/\fa, xM)$ are minimax. Hence from the exactness of 
	$0\to xM\to M\to M/{xM}\to 0$ we get that 
	$\Hom_A(R/\fa,T)$ and $\Ext_R^1(R/\fa, T)$, where $T=M/{xM}$, are minimax modules. 
	An argument similar to that one, we used to show that 
	$\lc_\fa^i(N)$ is $\fa$-cominimax,  for all $i$,  shows that $\lc_\fa^i(T)$ is $\fa$-cominimax, for 
	all $i$.

	Consider the homomorphism  $ f = x 1_M$, so $N=\Ker f$ and 
	$T=\Coker f$.
	We have shown that
	${\lc}^i_\fa(\Ker f ) $ and 
	${\lc}^i_\fa(\Coker f )$ are cominimax with respect 
	to $\fa$ for each $i$.
	By Proposition \ref{melc}  the class of $\fa$-cominimax modules, which are modules over  
	$\bar R$ annihilated by $x$ 
	constitute  a Serre subcategory of the category of $R$--modules. 
	Hence it follows from \cite[Corollary 3.2]{Mel} that for all $i$ the modules 
	$\Ker{{\lc}^i_\fa (f)} $ and  $\Coker{ {\lc}^i_\fa (f) }$ belong 
	to the same category. 
	Since $x\in\fa$ The criterion Lemma \ref{c} implies that 
	${\lc}^i_\fa (M )$ is $\fa$-cominimax,  for all $i$.
\end{proof}


\begin{thm}\label{cof3}
	Let $M$ be an $R$-module and suppose one of the following cases holds:
	\begin{itemize}
		\item[(a)] $\ara (\fa)\leq 1$;
		\item[(b)]  $\dim R/\fa \leq 1$;
		\item[(c)]  $\dim R\leq 2$.
	\end{itemize}
	Then,  $\lc^{i}_\fa(M)$ is $\fa$-cominimax for all $i$ if and only if $\Ext^i_R(R/\fa ,M)$ is a minimax $R$-module for all $i$.
\end{thm}

\begin{proof}
	The	case (c) follows by Theorem \ref{T:HE2}. 
	
	In the cases (a) and (b), suppose  $\lc^{i}_\fa(M)$ is $\fa$-cominimax for all $i$. It follows from \cite[Proposition 3.7]{ANV} that $\Ext^i_R(R/\fa ,M)$ is a minimax $R$-module for all $i$. 
	
	To prove the converse  we use induction on $i$. Let $i=0$. From the exact 
	
	\begin{center}
		$0\lo \G_{\fa}(M)\lo M \lo M/\G_{\fa}(M)
		\lo 0,$
	\end{center}
	
	we obtain that ${\Hom}_R(R/{\fa},\G_{\fa}(M))$ and ${\Ext}^{1}_R(R/{\fa},\G_{\fa}(M))$ are minimax. Now, it follows by \cite[Theorem 2.8]{A2} In the case (a) and \cite[Lemma 2.4]{KA} In the case (b) that $\G_{\fa}(M)$ is $\fa$-cominimax. It follows also that for all $i$ the $R$-modules ${\Ext}^{i}_R(R/{\fa},M/\G_{\fa}(M))$ are minimax.  Let $i> 0$ and the case $i-1$ is settled. Consider the exact sequence
	
	\begin{center}
		$0\lo M/\G_{\fa}(M)\lo E \lo L
		\lo 0,\,\,\,\,\,\,\,\,(\star)$
	\end{center}
	
	in which $E$ is an injective $\fa$-torsion free module. It is easy to see that $\lc_{\fa}^{i}(E)=0={\Ext}^{i}_R(R/{\fa},E)$ for all $i\ge 0$. Now, using the exact sequence $(\star)$, we easily get the isomorphisms  
	
	\begin{center}
		$\lc_{\fa}^{i}(L)\cong \lc_{\fa}^{i+1}(M/\G_{\fa}(M))\cong \lc_{\fa}^{i+1}(M)$.
		
	\end{center}
	
	\begin{center}
		\text{and} 	
	\end{center}	 
	
	\begin{center}
		${\Ext}^{i}_R(R/{\fa},L)\cong {\Ext}^{i+1}_R(R/{\fa},M/\G_{\fa}(M))$.
	\end{center}
	
	Hence the hypothesis is satisfied by $L$. This completes the inductive step. 
\end{proof}






\begin{thm}\label{cof4}
	Let $M$ be an $R$-module and suppose one of the following cases holds:
	\begin{itemize}
		\item[(a)] $\ara (\fa)\leq 1$;
		\item[(b)]  $\dim R/\fa \leq 1$;
		\item[(c)]  $\dim R\leq 2$.
	\end{itemize}
	Then, for any finite $R$-module $X$, $\lc^{i}_\fa(X,M)$ is $\fa$-cominimax for all $i\geq 0$ if and only if $\Ext^i_R(R/\fa ,M)$ is minimax $R$-module  for all $i\geq 0$.
\end{thm}

\begin{proof}
	First suppose for any finite $R$-module $X$, $\lc^{i}_\fa(X,M)$ is $\fa$-cominimax for all $i\geq 0$. Let $X=R$, then it follows by Theorem \ref{cof3}, that $\Ext^i_R(R/\fa ,M)$ is minimax $R$-module  for all $i\geq 0$.
	
	To prove the converse we use the induction on $i$. Let $i=0$, then it follows by Lemma \ref{1} (a) that
	\begin{center}
		$\lc^{0}_\fa(X,M)=\Gamma_{\fa}(X, M)\cong \Hom_{R}(X, \Gamma_\fa (M)).$
	\end{center}	
	
	Since $\Gamma_{\fa}(M)$ is $\fa$-cominimax by Theorem \ref{cof3},
	so the assertion follows by Corollary \ref{torextlc}, \cite[Corollary 2.11]{A2} and \cite[Corollary 2.8]{I}. Now assume that $i> 0$ and that the claim holds for $i-1$. Since $\G_{\fa}(M)$ is $\fa$-cominimax and $\Ext^i_R(R/\fa ,M)$ is minimax $R$-module for all $i\ge 0$. Using the short exact sequence
	
	\begin{center}
		$0\lo \G_{\fa}(M)\lo M \lo M/\G_{\fa}(M)
		\lo 0,$
	\end{center}
	
	it is easy to see that  the $R$-modules ${\Ext}^{i}_R(R/{\fa},M/\G_{\fa}(M))$ are minimax for all $i\ge 0$. Now, by applying the derived functor $\Gamma_{\fa}(X,-)$ to the same short exact sequence and using Lemma \ref{1} (b), we obtain the long exact sequence 
	
	\begin{center}
		$\dots \lo \Ext_R^i(X,\G_{\fa}(M))\overset{f_i} 
		\lo \lc^{i}_\fa(X,M)\overset{g_i}\lo \lc^{i}_\fa(X,M/\G_{\fa}(M))\overset{h_i}\lo \Ext_R^i(X,\G_{\fa}(M))\overset{f_{i+1}}\lo\lc^{i+1}_\fa(X,M)\lo\dots$
	\end{center}
	
	which yields short exact sequences 
	
	\begin{center}
		$0\lo \Ker{f_i}\lo \Ext_R^i(X,\G_{\fa}(M)) \lo \Image{f_i}
		\lo 0,$
	\end{center}
	
	\begin{center}
		$0\lo \Image{f_i}\lo \lc^{i}_\fa(X,M) \lo \Image{g_i}
		\lo 0$
	\end{center}
	
	\begin{center}
		$\text{and}$
	\end{center}
	
	\begin{center}
		$0\lo \Image{g_i}\lo \lc^{i}_\fa(X,M/\G_{\fa}(M)) \lo \Ker{f_{i+1}}
		\lo 0.$
	\end{center}
	
	Since $\Ext_R^i(X,\G_{\fa}(M))$	is  $\fa$-cominimax $R$-module for all $i\ge 0$ by Corollary \ref{torextlc}, \cite[Corollary 2.11]{A2} and \cite[Corollary 2.8]{I}, it follows by definition that, $\lc^{i}_\fa(X,M)$ is $\fa$-cominimax $R$-module if and only if $\lc^{i}_\fa(X,M/\G_{\fa}(M))$ is  $\fa$-cominimax $R$-module for all $i\ge 0$. Therefore it suffices to show that $\lc^{i}_\fa(X,M/\G_{\fa}(M))$ is  $\fa$-cominimax $R$-module for all $i\ge 0$. To this end consider the exact sequence 
	
	\begin{center}
		$0\lo M/\G_{\fa}(M)\lo E \lo L
		\lo 0,\,\,\,\,\,\,\,\,(\dagger)$
	\end{center}
	
	in which $E$ is an injective $\fa$-torsion free module. Since 
	$\G_{\fa}(M/\G_{\fa}(M))= 0= \G_{\fa}(E)$, thus $\Hom_R(R/\fa,E)=0$ and $\G_{\fa}(X, M/\G_{\fa}(M))= 0= \G_{\fa}(X, E)$ by Lemma \ref{1} (a). Applying the derived functors of $\Hom_R(R/\fa,-)$ and $\Gamma_{\frak{a}}(X,-)$ to the short exact sequence $(\dagger)$
	we obtain, for all $i> 0$, the isomorphisms
	
	\begin{center}
		$\lc^{i-1}_\fa(X, L)\cong \lc^i_\fa(X,M/\G_{\fa}(M))$.
		
	\end{center}
	
	\begin{center}
		\text{and} 	
	\end{center}	 
	
	\begin{center}
		${\Ext}^{i-1}_R(R/{\fa},L)\cong {\Ext}^{i}_R(R/{\fa},M/\G_{\fa}(M))$.
	\end{center}
	
	From what has already been proved, we conclude that  ${\Ext}^{i-1}_R(R/{\fa},L)$ is minimax for all $i>0$. Hence $\lc^{i-1}_\fa(X, L)$ is $\fa$-cominimax by induction hypothesis for all $i>0$, which yields that $\lc^i_\fa(X,M/\G_{\fa}(M))$ is $\fa$-cominimax for all $i\ge 0$, this completes the inductive step. 
\end{proof}


\begin{lem}\label{mar}
	
	Let $X$ be a finitely generated $R$-module and $M$ be an arbitrary $R$-module. Let $t$ be a non-negative integer such that ${\Ext}_R^i(X,M)$ is minimax for all $0\leq i\leq t$. Then for 
	any finitely generated $R$-module $L$ with
	$\Supp_R(L)\subseteq \Supp_R(X)$, ${\Ext}_R^i(L,M)$ is minimax for all integer $0\leq i\leq t$.
\end{lem}

\begin{proof}
	Use the method of proof of \cite[Proposition 1]{DM}.
\end{proof}	


The following corollary is our main result of this section which is a characterization of  cominimax local cohomology and generalized local cohomology modules under the assumptions (a) $\ara (\fa)\leq 1$, (b) $\dim R/\fa \leq 1$ and (c) $\dim R\leq 2$.

\begin{cor}\label{min}
	Let $M$ be an $R$-module and suppose one of the following cases holds:
	\begin{itemize}
		\item[(a)] $\ara (\fa)\leq 1$;
		\item[(b)]  $\dim R/\fa \leq 1$;
		\item[(c)]  $\dim R\leq 2$.
	\end{itemize}
	Then the following conditions
	are equivalent: 
	\begin{itemize}
		\item[(i)] $\Ext^i_R(R/\fa ,M)$ is minimax for all $i$.
		\item[(ii)] $\Ext^i_R(R/\fa ,M)$ is minimax for $i=0,1$ in the cases {\rm(}a{\rm)} and {\rm(}b{\rm)} {\rm(}resp. for $i=0,1,2$  in the case {\rm(}c{\rm)}{\rm)};
		\item[(iii)] $\lc^{i}_\fa(M)$ is  $\fa$-cominimax  for all $i$;
		\item[(iv)]  $\lc^{i}_\fa(X,M)$ is  $\fa$-cominimax
		for all $i$ and for any finite $R$-module $X$;
		\item[(v)] $\Ext^{i}_R(X,M)$ is  minimax  for all $i$ and for any finite $R$-module $X$ with $\Supp_R(X)\subseteq \V(\fa)$;
		\item[(vi)] $\Ext^{i}_R(X,M)$ is  minimax  for all $i$ and for some finite $R$-module $X$ with $\Supp_R(X)=\V(\fa)$;
		\item[(vii)] $\Ext^{i}_R(X,M)$ is  minimax  for $i=0,1$ in the cases {\rm(}a{\rm)} and {\rm(}b{\rm)} {\rm(}resp. for $i=0,1,2$  in the case {\rm(}c{\rm)}{\rm)} and for any finite $R$-module $X$ with $\Supp_R(X)\subseteq \V(\fa)$;
		\item[(viii)] $\Ext^{i}_R(X,M)$ is  minimax  for $i=0,1$ in the cases {\rm(}a{\rm)} and {\rm(}b{\rm)} {\rm(}resp. for $i=0,1,2$  in the case {\rm(}c{\rm)}{\rm)} and for some finite $R$-module $X$ with $\Supp_R(X)=\V(\fa)$.
	\end{itemize}
\end{cor}

\begin{proof}
	In order to prove (i)$\Leftrightarrow$(ii),  use \cite[Theorem 2.5]{A2} and \cite[Lemma 2.4]{KA}, in the cases (a) and (b) and use Theorem \ref{T:HE2}, in the case (c).
	
	(i)$\Leftrightarrow$(iii) follows by Theorem \ref{cof3}.
	
	(i)$\Leftrightarrow$(iv) follows by Theorem \ref{cof4}.
	
	In order to prove (i)$\Leftrightarrow$(v) and (ii)$\Leftrightarrow$(vii) use  \cite[Lemma 2.1]{A2}.
	
	(v)$\Rightarrow$(vi) and (vii)$\Rightarrow$(viii) are trivial.
	
	In order to prove (vi)$\Rightarrow$(v) and (viii)$\Rightarrow$(vii), let $L$ be a finitely generated $R$-module with $\Supp_R(L)\subseteq \V(\fa)$. Then $\Supp_R(L)\subseteq \Supp_R(X)$. Now the assertion follows by Lemma \ref{mar}.
\end{proof}


\begin{cor}
	Let $M$ be a minimax $R$-module, $X$ a finite $R$-module and $\fa$ an ideal of $R$.  Suppose one of the following cases holds:
	\begin{itemize}
		\item[(a)] $\ara (\fa)\leq 1$;
		\item[(b)]  $\dim R/\fa \leq 1$;
		\item[(c)]  $\dim R\leq 2$.
	\end{itemize}	
\end{cor}

Then $\lc^{i}_\fa(M)$ and $\lc^{i}_\fa(X,M)$ are $\fa$-cominimax
for all $i$.

\begin{proof}
	Follows from Corollary \ref{min}.	
\end{proof}

Now, It is natural to ask and offer the following question and a problem for the further research.\\

\noindent {\bf Question:} Let $R$ be a commutative Noetherian ring with non-zero identity and
$\fa$ an ideal of $R$. Is the charaterization  in Corollary \ref{min} true, when we change $\ara \fa\leq 1$ with $\cd(\fa,R)\leq 1$ in the case (a)?




\bibliographystyle{amsplain}

\begin{thebibliography}{9}




\bibitem{A3}
M, Aghapournahr,  {\it Cominimaxness  of  certain general local cohomology modules},  Hacet. J. Math. Stat., {\bf 18} (2019), 1--13.


\bibitem{A1}
M. Aghapournahr, {\it Cominimaxness of local cohomology modules}, Czechoslovak. Math. J., {\bf 69} (2019), 75--86.	







\bibitem{A2}
M. Aghapournahr, {\it Abelian category of cominimax and weakly cofinite modules}, Taiwanese J. Math., {\bf 20} (2016), 1001--1008.	 





\bibitem{AB}
M. Aghapournahr and K. Bahmanpour, {\it Cofiniteness of weakly Laskerian local	cohomology modules}, Bull. Math. Soc. Sci. Math. Roumanie, \textbf{105}
(2014), 347-356.

\bibitem{AM}
M. Aghapournahr and L. Melkersson, {\it Finiteness properties of minimax and coatomic local cohomology modules}, Arch. Math., \textbf{94} (2010), 519--528.


\bibitem{ATV}
M. Aghapournahr, A. J. Taherizadeh and A. Vahidi,
{\it Extension functors of local cohomology modules},
Bull. Iranian Math. Soc., \textbf{37} (2011), 117-134.







\bibitem{ANV}
J. Azami, R. Naghipour, and B. Vakili, {\it Finiteness properties of local cohomology modules
	for a-minimax modules}, Proc. Amer. Math. Soc., \textbf{137} (2009), 439-448.





\bibitem{BN}
K. Bahmanpour and  R. Naghipour, {\it Cofiniteness of local cohomology
	modules for ideals of small dimension},  J. Algebra, \textbf{321} (2009),
1997--2011.

\bibitem{BN2}
K. Bahmanpour and  R. Naghipour, {\it On the cofiniteness of local cohomology
	modules},  Proc. Amer. Math. Soc., \textbf{136} (2008), 2359-2363.






\bibitem{BNS1}
K. Bahmanpour, R. Naghipour and M. Sedghi, {\it On the finiteness of Bass numbers of local cohomology modules and cominimaxness}, Houston J. Math.,  {\bf 40} (2014), 319--337.
















\bibitem{BSh}  M. P. Brodmann and R. Y. Sharp ocal cohomology: An algebraic introduction with geometric
applications, Cambridge. Univ. Press, 1998.


\bibitem{BH} W. Bruns and J. Herzog,  Cohen Macaulay Rings,
Cambridge Studies in Advanced Mathematics, Vol. 39, Cambridge Univ. Press,
Cambridge, UK, 1993.




\bibitem{DM}
D. Delfino and T. Marley, {\it Cofinite modules and local cohomology}, J. Pure
Appl. Algebra, {\bf 121} (1997), 45--52.









\bibitem{DFT}
K. Divaani-Aazar, H. Faridian and M. Tousi,  {\it A new outlook on cofiniteness},  Kyoto J. Math. to appear.


\bibitem{E}
E. Enochs,  {\it Flat covers and flat cotorsion modules}, Proc. Amer. Math. Soc., {\bf92} (1984), 179-184.


\bibitem{Gro}
A. Grothendieck, \emph{Cohomologie locale des faisceaux
	coh\'{e}rents et th\'{e}or\`{e}mes de Lefschetz
	locaux et globaux {\rm(}SGA2{\rm)}},
North-Holland, Amsterdam, 1968.




\bibitem{Har}
R. Hartshorne, {\it Affine duality and cofiniteness}, Invent. Math., {\bf9} (1970),
145-164.


\bibitem{HR}
D. Hassanzadeh-Lelekaami, H. Roshan-Shokalgourabi, {\it Extension functors of cominimax modules}, Comm. Algebra, {\bf 45} (2017), 621--629.

\bibitem{He}
J. Herzog, \emph{Komplexe, Aufl\"{o}sungen und Dualit\"{a}t in der lokalen Algebra},
Habilitationsschrift, Universitat Regensburg 1970.




\bibitem{HK}
C. Huneke and J. Koh,
\emph{Cofiniteness and vanishing of local cohomology
	modules}, Math. Proc. Cambridge Philos. Soc., \textbf{110} (1991),
421--429.




\bibitem{I}
Y. Irani, {\it Cominimaxness with respect to ideals of dimension one}, Bull. Korean Math. Soc., \textbf{54} (2017), 289-298.




\bibitem{KA}
H. Karimirad and M. Aghapournahr, {\it Cominimaxness with respect to ideals of dimension two and local cohomology},  J. Algebra Appl., \textbf{20} (2021), 2150081. 










\bibitem{MV}
T. Marley and J. C. Vassilev, {\it Cofiniteness and associated primes of local
	cohomology modules}, J. Algebra, {\bf 256} (2002), 180--193.

\bibitem{Mat}
H. Matsumura,  Commutative ring theory, Cambridge Univ. Press, Cambridge,
UK, 1986.

\bibitem{Mel1} L. Melkersson,  {\it Cofiniteness with respect to ideals of dimension one}, J.
Algebra, {\bf372}  (2012), 459-462.

\bibitem{Mel} L. Melkersson, {\it Modules cofinite with respect to an ideal}, J.
Algebra, {\bf 285} (2005), 649--668.



\bibitem{NS}
M. Nazari and R. Sazeedeh, \emph{Cofiniteness with respect to two ideals and local cohomology}, Algebr. Represent. Theor.,
\textbf{22} (2019), 375--385.

\bibitem{VA}
A. Vahidi and M. Aghapournahr, {\it Some results on generalized   local cohomology modules}, Comm. Algebra,
\textbf{43} (2015), 2214-2230.












\bibitem{Zi}
T. Zink,
\emph{Endlichkeitsbedingungen f\"{u}r Moduln \"{u}ber einem
	Noetherschen Ring},
Math. Nachr., \textbf{164} (1974), 239--252.




\bibitem{Zrmm}
H. Z\"{o}schinger, \emph{Minimax Moduln}, J. Algebra,
\textbf{102} (1986), 1--32.






\end{thebibliography}

\end{document}